\documentclass[reqno,12pt]{amsart}
\usepackage{DocumentPreamble}

\AddTitle{A sufficient condition for generalized spectral characterization of graphs with loops}
\AddAuthor{Alexander Van Werde}
\email{\href{mailto:a.van.werde@uni-muenster.de}{a.van.werde@uni-muenster.de}}
\address{University of Münster, Germany}

\AddMSC{05C50}
\AddMSC{15B36}
\AddMSC{11C20}

\AddKeyword{cospectral}
\AddKeyword{walk matrix}
\AddKeyword{square-free}
\AddKeyword{integral matrix}

\begin{document}
\begin{abstract}
    Sufficient conditions for a simple graph to be characterized up to isomorphism given its spectrum and the spectrum of its complement graph are known due to Wang and Xu. 
    This note establishes a related sufficient condition in the presence of loops: if the walk matrix has square-free determinant, then the graph is characterized by its generalized spectrum.
    The proof includes a general result about symmetric integral matrices.
\end{abstract}
\maketitle
\vspace{-1.5em}
\section{Introduction}
Let $G$ be an undirected graph, possibly including self-loops.
The spectrum of its adjacency matrix encodes significant combinatorial information, and one may wonder whether this is sufficient to characterize the graph up to isomorphism.
This question dates back to the 1950's when it was posed by G\"unthard and Primas in a paper on quantum chemistry \cite{gunthard1956zusammenhang}. 
It can also be viewed a graph-theoretic variant to Kac's  famous question ``Can one hear the shape of a drum?'' \cite{kac1966can}.

A significant open problems in the field is Haemers' conjecture, which posits that almost all graphs are characterized by their adjacency spectrum \cite{van2003graphs,haemers2016almost}. 
Unfortunately, it is generally difficult to determine whether a given graph is characterized by its adjacency spectrum. 
While there are many methods that one can try to use to produce another graph with the same spectrum \cite{schwenk1973almost,abiad2024switching,godsil1982constructing,wang2019cospectral}, there is no easy way to verify that a graph is characterized by its adjacency spectrum aside from exhaustive enumeration, unless the graph belongs to some special class whose structure can be exploited \cite{doob2002complement,shen2005graph,haemers2008spectral}. 

\pagebreak[3]
Notably, however, Wang and Xu \cite{wang2006sufficient,wang2017simple} found that tractable sufficient conditions can be established if one is also given the spectrum of the complement graph.
By a theorem of Johnson and Newman \cite{johnson1980note} it is equivalent to consider the bivariate polynomial $\varphi_G(\lambda,t) \de \det(\lambda\bI - \bA  - t \bJ)$ where $\bA$ is the adjacency matrix, $\bJ$ is the all-ones matrix, and $\bI$ is the identity matrix. 
Two graphs $G,H$ are called \emph{$\bbR$-cospectral} if $\varphi_G(\lambda,t) = \varphi_H(\lambda,t)$ for all $\lambda,t\in \bbR$. 
The graph $G$ is said to be \emph{characterized by its $\bbR$-spectrum} if any such $H$ is isomorphic to $G$. 
\pagebreak[3]

The \emph{walk matrix} is defined column-wise as $\bW \de [e,\bA e,...,\bA^{n-1}e]$ with $e \de (1,\ldots,1)^{\T}$ the all-ones vector.
An integer $M\in \bbZ$ is \emph{square free} if there does not exist any prime $p$ with $p^2\mid M$. 
\begin{theorem}[{Wang \cite{wang2017simple}}]\label{thm: Wang}
    Suppose that $G$ is simple and that $\det(\bW)/2^{\lfloor n/2 \rfloor}$ is odd and square-free integer. 
    Then, $G$ is characterized by its $\bbR$-spectrum.  
\end{theorem}

We establish a variation on Theorem \ref{thm: Wang} for graphs with loops: 

\begin{theorem}\label{thm: WithLoops}
    Consider a graph $G$ that possibly has loops and suppose that $\det(\bW)$ is square-free. Then, $G$ is characterized by its $\bbR$-spectrum. 
\end{theorem}

While Theorem \ref{thm: WithLoops} also allows simple graphs, it is vacuous in this case. 
Indeed, it can be shown that  $2^{\lfloor n/2 \rfloor}$ always divides the determinant of the walk matrix for simple graphs \cite[Lemma 14]{wang2013generalized}.
For graphs with loops, it can occur that $2^2$ does not divide the determinant; see Figure \ref{fig: ExampleGraphWalk} for example.
Adding loops hence has the nice feature that it eliminates the need to treat the prime $p=2$ separately.

\begin{figure}[t]
    \centering
\begin{minipage}[t]{.34\textwidth}
    \centering
    \large 
    \ \ \ \ \ \textbf{ Graph}\\ 
    \includegraphics[width = \textwidth]{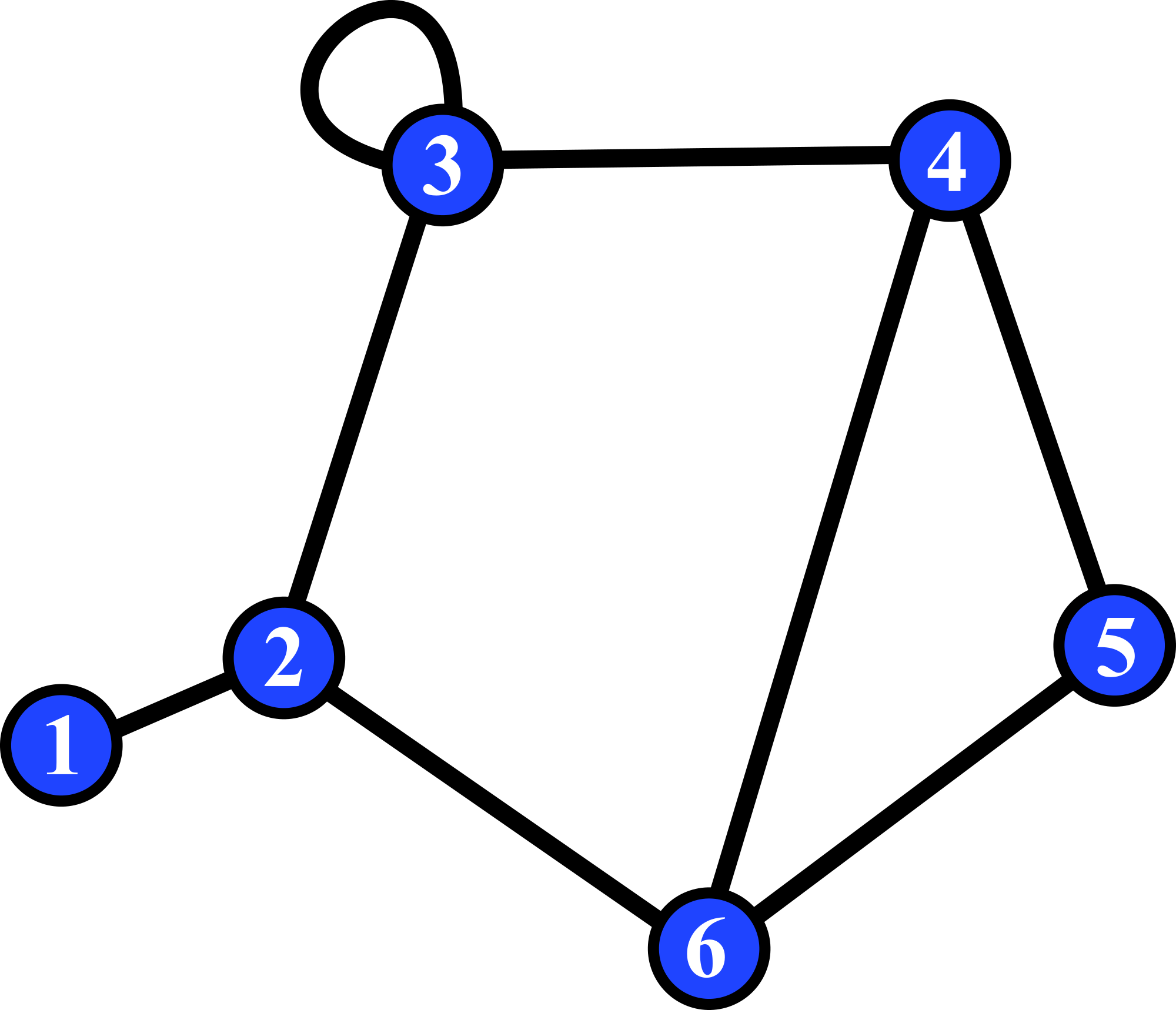}
\end{minipage}%
\hfill 
\begin{minipage}[t]{0.29\textwidth}
    \centering 
    \centering
    \large 
    \textbf{Adjacency}\\ 
    \ \\     
    \begingroup
    \setlength\arraycolsep{4pt}
    $\begin{pmatrix}
        \gray{0}&\mathbf{1}&\gray{0}&\gray{0}&\gray{0}&\gray{0}\\
        \mathbf{1}&\gray{0}&\mathbf{1}&\gray{0}&\gray{0}&\mathbf{1}\\
        \gray{0}&\mathbf{1}&\mathbf{1}&\mathbf{1}&\gray{0}&\gray{0}\\
        \gray{0}&\gray{0}&\mathbf{1}&\gray{0}&\mathbf{1}&\mathbf{1}\\
        \gray{0}&\gray{0}&\gray{0}&\mathbf{1}&\gray{0}&\mathbf{1}\\
        \gray{0}&\mathbf{1}&\gray{0}&\mathbf{1}&\mathbf{1}&\gray{0}  
    \end{pmatrix}$
    \endgroup
\end{minipage}%
\hfill 
\begin{minipage}[t]{0.34\textwidth}
    \centering
    \centering
    \large 
    \textbf{Walk matrix}
    \\ 
    \ \\ 
    \begingroup
    \setlength\arraycolsep{4pt}
    $\begin{pmatrix}
        1&1&3&7&20&52\\
        1&3&7&20&52&146\\
        1&3&9&24&67&180\\
        1&3&8&23&61&170\\
        1&2&6&16&44&120\\
        1&3&8&21&59&157
    \end{pmatrix}$
    \endgroup
    \end{minipage}
    \label{fig:prob3}
    \caption{An example of a graph with loops and associated matrices. 
    Note that the $ij$-th entry of the walk matrix counts walks of length $j-1$ starting at vertex $i$, explaining the terminology. 
    It here holds that $\det(\bW) = -3$. 
    Thus, the graph is characterized by its $\bbR$-spectrum due to Theorem \ref{thm: WithLoops}.}
    \label{fig: ExampleGraphWalk}
\end{figure}

Our proofs also enable generalizations of \Cref{thm: WithLoops} where the adjacency matrix may be replaced by an arbitrary symmetric integer matrices and the all-ones vector by an arbitrary integer vector; see \Cref{thm: SqfreeIntegral}. 
Aside from its independent interest, this general setting is helpful to simplify probabilistic analysis, which was our main motivation for the present paper. 
In particular, random matrix models that are expected to behave similarly to graphs with loops are analyzed in a forthcoming work of Lvov and the present author \cite{vanwerde2025}, leading to conjectures for the satisfaction frequency of \Cref{thm: SqfreeIntegral} when the integer vector is also taken to be random.
We return to this point in \Cref{conj: Satisfaction}.

\section{Proof of \texorpdfstring{Theorem \ref{thm: WithLoops}}{Theorem}}
We establish a result in a more general framework. 
Consider a vector $\zeta\in \bbZ^n$ and an integral matrix $\bX \in \bbZ^{n\times n}$ that is symmetric $\bX = \bX^{\T}$. 
Two pairs $(\bX,\zeta)$ and $(\bY, \eta)$ are \emph{isomorphic up to sign} if there exists a signed permutation $\bS$ such that $\bS\bX \bS^{\T} = \bY$ and $\bS \zeta = \eta$. 
A \emph{signed permutation} is a matrix of the form $\bS = \bP \bD$ with $\bP \in \{0,1 \}^{n\times n}$ a permutation and $\bD$ a diagonal matrix with diagonal entries $\pm 1$. 

\pagebreak[4]
Note that the polynomial $\Phi_{\bX,\zeta}(\lambda,t) \de \det(\lambda\bI - \bX  - t\zeta  \zeta^{\T})$ is an invariant for this notion of isomorphism.\label{p: Def_Phi} 
The pair $(\bX,\zeta)$ is \emph{characterized by $\Phi$-spectrum} if every $(\bY,\eta)$ with $\Phi_{\bX,\zeta} = \Phi_{\bY,\eta}$ is isomorphic up to sign.
Observe that restricting this notion to adjacency matrices of graphs with $\zeta = e$ recovers $\bbR$-spectral characterization because nonnegative matrices are related through a signed permutation if and only if the signs can be taken to be positive, and because permutation corresponds to graph isomorphism. 
We establish a sufficient condition for $\Phi$-spectral characterization in Theorem \ref{thm: SqfreeIntegral} that recovers Theorem \ref{thm: WithLoops} as a special case.  

A condition related to Theorem \ref{thm: Wang} for a (simple) graph to be characterized by the $\Phi$-spectrum with a vector different from the all-ones vector was studied by Qiu, Ji, Mao, and Wang \cite{qiu2023generalized} motivated by the fact that Wang's condition \cite{wang2006sufficient} is not applicable to regular graphs. 
Further, conditions for integral matrices with diagonal zero have been considered by Miao, Li, and Zhang \cite{miao2026smith}. 
Symmetric integral matrices with nonzero diagonal, as are necessary in the presence of loops, have also been studied by Wang and Yu \cite{wang2016square}, who established conditions of a different nature using the discriminant of the characteristic polynomial instead of the determinant of the walk matrix.  
One can combine \cite{wang2016square} with our subsequent results to achieve more powerful sufficient conditions; see Remark \ref{rem: WangYuDiscriminant}.

Our proof takes inspiration from recent work by Qui, Wang, Wang, and Zhang \cite{qiu2023smith} who established a refinement of Wang's condition \cite{wang2006sufficient} with a new approach that is cleaner than previous works in the literature. 
We justify that the approach can also be made to work in the presence of loops, giving different arguments in some places; see also the discussion preceding Lemma \ref{lem: DistinctRoots}. 
Further, the full strength of \cite{qiu2023smith} was not needed for Theorem \ref{thm: WithLoops} and we developed alternative claims and proofs in a number of steps to achieve further simplifications. 

\subsection{Characterization in terms of orthogonal matrices}
The following characterization of $\Phi$-cospectrality is analogous to one for $\bbR$-cospectrality that was first discovered by Johnson and Newman \cite{johnson1980note}.
Alternative proofs and extensions have since also appeared  \cite{qiu2023generalized,wang2006sufficient,coutinho2021graph,wang2024graph}. 
We give a proof for completeness. 
\begin{lemma}\label{lem: PhiOrthogonal}
    Two pairs $(\bX,\zeta)$ and $(\bY, \eta)$ are $\Phi$-cospectral if and only if there exists an orthogonal matrix $\bQ \in \bbR^{n\times n}$ with $\bQ \zeta = \eta$ and $\bQ \bX \bQ^{\T} = \bY$. 
\end{lemma}
\begin{proof} 
    That the existence of an orthogonal matrix $\bQ$ with $\bQ \bX \bQ^{\T} = \bY$ and $\bQ \zeta = \eta$ implies $\Phi$-cospectrality is clear, since we then have  
\begin{equation}
    \det\bigl(\lambda\bI -\bX - t\zeta \zeta^{\T} \bigr) = \det\bigl(\bQ ( \lambda\bI -\bX- t\zeta \zeta^{\T}) \bQ^{\T}\bigr) = \det( \lambda\bI -\bY- t \eta \eta^{\T}). 
\end{equation}
    Now assume $\Phi$-cospectrality, and let us establish the existence of some $\bQ$.

    Consider some arbitrary $\lambda$ that is not an eigenvalue of $\bX$, nor of $\bY$. 
    Using that $\bX -\lambda\bI$ is then invertible and that $\det(\bI - vw^{\T} ) = 1 - v^{\T} w$ for every $v,w\in \bbR^n$, 
    \begin{align}
        \Phi_{\bX,\zeta}(\lambda,t) &=  \det( \lambda\bI - \bX)\det\bigl(\bI - t( \lambda\bI - \bX)^{-1} \zeta \zeta^{\T} \bigr)\nonumber \\ 
        &= \det( \lambda\bI - \bX) (1 - t\zeta^{\T} (\lambda \bI -\bX)^{-1} \zeta).\label{eq:TallHome}
    \end{align}
    In other words, $\Phi_{\bX,\zeta}(\lambda,t) = \Phi_{\bX,\zeta}(\lambda,0) (1 - t\zeta^\T (\lambda\bI - \bX )^{-1} \zeta )$. 
    Using the same identity applied to $\Phi_{\bY,\eta}$ and the assumption of $\Phi$-cospectrality now yields that 
    \begin{equation}
         \zeta^\T (\lambda\bI - \bX )^{-1} \zeta = \eta^\T ( \lambda\bI - \bY)^{-1} \eta. \label{eq:SneakyCup}
    \end{equation}
    Let $\sigma \subseteq\bbR$ be the spectrum of $\bX$, which is also the spectrum of $\bY$. 
    Further, for every $\mu \in \sigma$ let $V(\bX,\mu) \subseteq\bbR^n$ denote the corresponding eigenspace of $\bX$, and let $\bP_{V(\bX,\mu)}  $ be the associated orthogonal projection.  
    Then, $(\lambda\bI - \bX)^{-1} = \sum_{\mu \in\sigma} (  \lambda - \mu )^{-1}\bP_{V(\bX,\mu)}$ and a similar identity applies to $(\lambda \bI - \bY )^{-1}$. 
    Hence, by \eqref{eq:SneakyCup}, 
    \begin{equation}
        \sum_{\mu \in \sigma } \frac{\Vert \bP_{V(\bX,\mu)} \zeta \Vert^2}{\lambda - \mu} =  \sum_{\mu \in \sigma } \frac{\Vert  \bP_{V(\bY,\mu)} \eta \Vert^2}{\lambda - \mu}.\label{eq:FairTrunk} 
    \end{equation}
    This implies that $\Vert \bP_{V(\bX,\mu)}  \zeta \Vert = \Vert  \bP_{V(\bY,\mu)}  \eta \Vert$ for every $\mu \in \sigma$. 

    The assumption of $\Phi$-cospectrality implies, in particular, that $\dim(V(\bX,\mu)) = \dim(V(\bY,\mu))$ for every $\mu \in \sigma$ since algebraic multiplicity is equal to geometric multiplicity for symmetric matrices.
    It now follows from the fact that $\Vert \bP_{V(\bX ,\mu)} \zeta \Vert = \Vert  \bP_{V(\bY ,\mu) } \eta \Vert$ that there exist orthogonal operators $\bQ_\mu: V(\bX,\mu) \to V(\bY,\mu)$ satisfying $\bQ_{\mu} \bP_{V(\bX,\eta)} \zeta = \bP_{V(\bY,\mu)}\eta$. 
    Fix these $\bQ_{\mu}$ from here on.
    Then, since $ \oplus_{\mu \in \sigma} V(\bX,\mu) $ and $\oplus_{\mu \in \sigma} V(\bY,\mu)$ are decomposition of $\bbR^n$ into orthogonal subspaces, there exists a unique orthogonal transformation $\bQ:\bbR^n\to \bbR^n$ with $\bQ(V(\bX,\mu)) = V(\bY,\mu)$ and $\bQ\mid_{V(\bX,\mu)} = \bQ_{\mu}$. 
    This transformation satisfies that for every $\mu \in \sigma$,
    \begin{equation}
        \bQ \bP_{V(\bX,\mu)} \bQ^{\T} = \bP_{V(\bY,\mu)} \quad\textnormal{and}\quad  \bQ \bP_{V(\bX,\mu)} \zeta  = \bP_{V(\bY,\mu)}\eta.\label{eq:MildDragon}   
    \end{equation} 
    Combine \eqref{eq:MildDragon} with the fact that $\bX = \sum_{\mu \in \sigma } \mu \bP_{V(\bX,\mu)}$ and $\zeta = \sum_{\mu \in \sigma}\bP_{V(\bX,\mu)} \zeta$ to conclude that $\bQ \bX \bQ^{\T} = \bY$ and $\bQ \zeta = \eta$. 
    
\end{proof}

\subsection{Constraints on the orthogonal matrix}
From here on, let us fix $\Phi$-cospectral $(\bX,\eta)$ and $(\bY,\eta)$ and let $\bQ$ be an orthogonal matrix as in Lemma \ref{lem: PhiOrthogonal}. 
Further, define the \emph{walk matrix of $(\bX,\zeta)$}, 
\begin{equation}
        \bW_{\bX,\zeta} = \bigl[\zeta,\, \bX \zeta, \ldots, \bX^{n-1} \zeta  \bigr]. \label{eq: Def_WXzeta}    
\end{equation} 
Our goal is to deduce constraints on $\bQ$ based on the properties of $\bW_{\bX,\zeta}$. 
\begin{lemma}\label{lem: NonzeroDet_RationalQ}
    Assume that $\det(\bW_{\bX,\zeta}) \neq 0$. 
    Then, it holds that $\bQ\in \bbQ^{n\times n}$.
\end{lemma}
\begin{proof}
    By using that $\bQ\zeta = \eta$ and repeatedly using that $\bQ \bX = \bY \bQ$, 
    \begin{equation}
        \bQ \bW_{\bX, \zeta} = \bW_{\bY, \eta}.\label{eq: QWW}
    \end{equation}
    The determinant being nonzero means that $\bW_{\bX,\zeta}$ is invertible over $\bbQ$.
    Hence, $\bQ =  \bW_{\bY, \eta}\bW_{\bX,\zeta}^{-1}$ is indeed rational.     
\end{proof}

From here on, we adopt as a running assumption that $\det(\bW_{\bX,\zeta}) \neq 0$.  
The \emph{level} of the rational orthogonal matrix $\bQ$ is the least integer $\ell$ with the property that $\ell \bQ \in \bbZ^{n\times n}$. 
It holds that $\ell = 1$ if and only if $\bQ$ is a signed permutation, so the goal becomes to bound $\ell$.  

\begin{lemma}\label{lem: EllDividesDn}
    The level $\ell$ of $\bQ$ satisfies $\ell \mid \det(\bW_{\bX,\eta})$.
\end{lemma} 
\begin{proof} 
    The adjoint formula for the inverse implies that $\det(\bW_{\bX,\zeta})  \bW_{\bX,\zeta}^{-1}$ is an integral matrix.
    That $\bQ =  \bW_{\bY,\eta}\bW_{\bX,\zeta}^{-1}$ by \eqref{eq: QWW} then implies that also $\det(\bW_{\bX,\zeta})\bQ \in \bbZ^{n\times n}$. The desired result now follows by the minimality of $\ell$. 
\end{proof}

Lemma \ref{lem: EllDividesDn} implies that every prime divisor of $\ell$ also has to be a prime divisor of $\det(\bW_{\bX,\zeta})$. 
This is however not sufficient to ensure that $\ell = 1$ because the determinant typically has nontrivial prime divisors.

Given an integral matrix $\bM \in \bbZ^{n\times n}$ let $\rank_p(\bM)$ denote the rank of $\bM \bmod p$ over $\bbF_p \de \bbZ/p\bbZ$. 
Then, $p\mid \det(\bW_{\bX,\zeta})$ if and only if $\operatorname{rank}_p(\bW_{\bX,\zeta}) \leq n-1$.
The following Lemma \ref{lem: StrongerLevelBound} hence implies that Theorem \ref{thm: WithLoops} would follow if we could show that every vector in the kernel of $\bW_{\bX,\zeta}$ over $\bbF_p$ is also in the kernel of $\bW_{\bY,\eta}$:
\begin{lemma}\label{lem: StrongerLevelBound}
    Suppose that $\rank_p(\bW_{\bX,\zeta}) \leq n-1$. 
    Moreover, assume that every $V\in \bbZ^n$ with $\bW_{\bX,\zeta}V \equiv 0 \bmod p$ satisfies $\bW_{\bY, \eta} V \equiv 0 \bmod p$. 
    Then, $\ell \mid \det(\bW_{\bX,\zeta})/p$. 
\end{lemma}
\begin{proof}
    That $\rank_p(\bW_{\bX,\zeta}) \leq n-1$ means that there exist $v_0,\ldots,v_{n-1} \in \bbZ$ that are not all congruent to zero with $\sum_{i=0}^{n-1}v_i \bX^i \zeta \equiv 0 \bmod p$. 
    Possibly multiplying this relation with a power of $\bX$, we may assume that $v_{n-1} \not\equiv 0 \bmod p$. 
    Further, we can assume that $v_{n-1} = 1$.  
    At any rate, the following matrix is still in $\bbZ^{n\times n}$:
    \begin{equation}
        \tilde{\bW}_{\bX,\zeta} \de \biggl[\zeta, \bX \zeta, \ldots, \bX^{n-2}\zeta, \frac{1}{p} \sum_{i=0}^{n-1}v_i \bX^i \zeta  \biggr], \label{eq: Def_tildeW}
    \end{equation}
    Moreover, since $V \de (v_0,\ldots,v_{n-1})^{\T}$ is also in the kernel of $\bW_{\bY,\eta}$ modulo $p$ by assumption, the matrix $\tilde{\bW}_{\bY, \eta}$ defined similarly to \eqref{eq: Def_tildeW} is also an element of $\bbZ^{n\times n}$.

    Using that $v_{n-1} = 1$ and the basic properties of the determinant,  
    \begin{equation}
        \det(\tilde{\bW}_{\bX,\zeta}) = \frac{1}{p} \det(\bW_{\bX,\zeta}).\label{eq:OddPop} 
    \end{equation}
    In particular, the running assumption that $\det(\bW_{\bX,\zeta})\neq 0$ yields that $\tilde{\bW}_{\bX,\zeta}$ is invertible over $\bbQ$. 
    Now, using that $\bQ \bX = \bY \bQ$ and $\bQ\zeta = \eta$, we also have $\bQ\tilde{\bW}_{\bX,\zeta}= \tilde{\bW}_{\bY, \eta}$ and hence $\bQ = \tilde{\bW}_{\bY, \eta}\tilde{\bW}_{\bX,\zeta}^{-1}$. 
    Consequently, $\det(\tilde{\bW}_{\bX,\zeta})\bQ \in \bbZ^{n\times n}$ exactly as in the proof of Lemma \ref{lem: EllDividesDn}, and the desired result follows from \eqref{eq:OddPop}.  
\end{proof}
\subsection{Comparing the kernels of the walk matrices}
It will be useful to assume not only that $\operatorname{rank}_p(\bW_{\bX,\zeta}) \leq n-1$ but that even equality holds. 
This is justified by the following: 
\begin{lemma}\label{lem: PSquareRank}
    Suppose that $p \mid \det(\bW_{\bX,\zeta})$ but $p^2 \nmid \det(\bW_{\bX,\zeta})$. 
    Then, $\operatorname{rank}_p(\bW_{\bX,\zeta}) = n-1$ and $\operatorname{rank}_{p}(\bW_{\bY,\eta}) = n-1$. 
\end{lemma}
\begin{proof}
    That $p\mid \det(\bW_{\bX,\zeta})$ implies that $\rank_{p}(\bW_{\bX,\zeta}) \leq n-1$. 
    We show that the inequality can not be strict. 
    Suppose to the contrary that $\rank_p(\bW_{\bX,\zeta}) \leq n-2$. 
    Then, there exists $\bM \in \bbZ^{n\times n}$ with $\bM \bmod p$ invertible over $\bbF_p$ such that $\bW_{\bW_{\bX,\zeta}}\bM$ has two columns congruent to $0$ modulo $p$. 
    The multilinearity of the determinant then yields $p^2 \mid \det(\bW_{\bW_{\bX,\zeta}}\bM)$. 
    Here, $\det(\bW_{\bW_{\bX,\zeta}}\bM ) =  \det(\bW_{\bW_{\bX,\zeta}})\det(\bM)$ and $p\nmid\det(\bM)$ by the invertibility of $\bM$ over $\bbF_p$.
    Hence,  $p^2 \mid \det(\bW_{\bX,\zeta})$, a contradiction. 
    
    We have shown that $\rank_p(\bW_{\bX,\zeta})= n-1$ if $p\mid \det(\bW_{\bX,\zeta})$ and $p^2 \nmid \det(\bW_{\bX,\zeta})$. 
    It also follows that $\rank_p(\bW_{\bY,\eta}) = n-1$ since \eqref{eq: QWW} yields that $\det(\bW_{\bX,\zeta}) = \pm\det(\bW_{\bY,\eta})$ using that $\det(\bQ) = \pm 1$ for any orthogonal matrix $\bQ$.  
\end{proof}

\pagebreak[3]
The following result generalizes \cite[Lemma 2.6]{qiu2023smith} whose proof uses that their graph is simple so that a coefficient of the characteristic polynomial is zero. 
This is helpful in \cite{qiu2023smith} because their proof relies on a system of equations involving these coefficients.
An extension to a case with nonzero trace also seems to be implicitly used in \cite[Corollary 1]{qiu2023smith}, but explicit argumentation was there omitted. 

Aside from making it explicit that an extension to a setting with nonzero trace is possible, the following alternative argument simplifies the proof by eliminating the need to chase coefficients through a system of equations altogether. 
\begin{lemma}\label{lem: DistinctRoots}    
    Suppose that $\rank_{p}(\bW_{\bX,\zeta}) = n-1$ and $\rank_{p}(\bW_{\bY,\eta}) = n-1$. 
    Further, assume that $\varphi_\bX(\lambda) \de \det(\lambda\bI - \bX)$ does not have two distinct roots in $\bbF_p$. 
    Then, $\bW_{\bX,\zeta} V \equiv 0 \bmod p$ and $\bW_{\bY,\eta}V \equiv 0 \bmod p$ have the same solutions $V$. 
\end{lemma}
\begin{proof} 
    That $\rank_p(\bW_{\bX,\zeta}) = n-1$ implies that any nonzero polynomial $\psi \in \bbF_p[\lambda]$ with $\psi(\bX) \zeta \equiv 0 \bmod p$ has degree no less than $n-1$. 
    Indeed, the vectors of the coefficients for $\psi(\lambda)$ and $\lambda \psi(\lambda)$ would otherwise induce two independent relation on the columns of the walk matrix. 
    In particular, the minimal polynomial of $\zeta$  for the matrix $\bX$, denoted $\mu_{\bX,\zeta}(\lambda)\in \bbF_p[\lambda]$, has degree no less than $n-1$. 
    
    Now consider some vector $V= (v_0,\ldots,v_{n-1})^{\T}$ and define a polynomial in $\bbF_p[\lambda]$ by $\psi_V(\lambda) \de \sum_{i=0}^{n-1} v_i \lambda^i \bmod p$. 
    That $\bW_{\bX,\zeta} V \equiv 0 \bmod p$ is then equivalent to the statement that $\psi_V(\bX)\zeta = 0$ in $\bbF_p^n$.
    Note that $\psi_V$ has degree $\leq n-1$ to conclude that $\bW_{\bX,\zeta} V \equiv 0 \bmod p$ if and only if $\psi_V =c \mu_{\bX,\zeta}$ for some $c\in \bbF_p$. 
    In particular, we have that $\operatorname{deg}(\mu_{\bX,\zeta})  = n-1$ since the assumption that $\rank_p(\bW_{\bX,\zeta}) \leq n-1$ implies that there exists at least one $V \not\equiv 0 \bmod p$ with $\bW_{\bX,\zeta} V \equiv 0 \bmod p$. 

    Similar conclusions apply if $\bX$ and $\zeta$ are replaced by $\bY$ and $\eta$, respectively. 
    It now suffices to show that $\mu_{\bX,\zeta} = \mu_{\bY,\eta}$.
    To this end, we next exploit that $\varphi_\bX = \varphi_{\bY}$ due to the running assumption of $\Phi$-cospectrality.

    The Cayley--Hamilton theorem yields that $\varphi_{\bX}(\bX) = 0$. 
    In particular, we have $\varphi_{\bX}(\bX)\zeta \equiv 0 \bmod p$ which means that $\mu_{\bX,\zeta}$ divides $\varphi_{\bX} \bmod p$. 
    Hence, considering that $\operatorname{deg}(\mu_{\bX,\zeta}) = n-1$ and $\operatorname{deg}(\varphi_{\bX}) = n$, there exists some $\lambda_0 \in \bbF_p$ with 
    \begin{equation}
        \varphi_\bX(\lambda) \equiv (\lambda - \lambda_0)\mu_{\bX,\zeta}(\lambda) \mod p.   
    \end{equation}
    Similarly, $\varphi_\bY(\lambda) \equiv (\lambda - \lambda_0')\mu_{\bY,\eta}(\lambda) \bmod p$ for some $\lambda_0' \in \bbF_p$. 
    Using that $\varphi_\bY = \varphi_\bX$ and that $\varphi_{\bX}$ has no two distinct roots here yields $\lambda_0 = \lambda_0'$. 
    Then also $\mu_{\bX,\zeta} = \mu_{\bY,\eta}$ since $\bbF_p[\lambda]$ is a unique factorization domain. 
    This concludes the proof.  
\end{proof}
The assumption on the roots of $\varphi_{\bX}$ that is required to apply Lemma \ref{lem: DistinctRoots} is unfortunately not always satisfied. 
To enhance the result's applicability, we can instead consider the matrices defined for $t\in \bbZ$ by  
\begin{equation}
    \sX_t \de \bX + t\zeta \zeta^{\T}, \qquad \sY_t \de \bY + t \eta \eta^{\T}.     
\end{equation}
These matrices are still $\Phi$-cospectral. 
Further, for every $i\geq 0$ direct computation yields $c_0,\ldots,c_{i-1}\in \bbZ[t]$ depending on $\bX,\zeta$ with $\sX_t^i \zeta = \bX^i \zeta + \sum_{j=0}^{i-1} c_j\bX^{j}\zeta$, which implies that $\bW_{\sX_t, \zeta}$ has the same determinant and rank as $\bW_{\bX, \zeta}$ for every fixed $t$.

\begin{corollary}\label{cor: t0Roots}
    Suppose that $\rank_{p}(\bW_{\bX,\zeta}) = \rank_{p}(\bW_{\bY,\eta}) = n-1 $.
    Further, assume that we are given some $t_0\in \bbZ$ such that $\varphi_{\sX_{t_0}}(\lambda)$ has no two distinct roots in $\bbF_p$. 
    Then, $\bW_{\sX_{t_0},\zeta} V \equiv 0 \bmod p$ and $\bW_{\sY_{t_0},\eta}V \equiv 0 \bmod p$ have the same solutions $V$. 
\end{corollary}
\begin{proof}
    This is immediate from  Lemma \ref{lem: DistinctRoots}. 
\end{proof}
We show in Lemma \ref{lem: DistinctRoots} that there exists a value of $t_0$ for which Corollary \ref{cor: t0Roots} is applicable, after which we conclude in Section \ref{sec: Sufficient}. 
First, some preparatory lemmas are required. 
Recall the definition of $\Phi_{\bX,\zeta}$ from the start of Section \ref{p: Def_Phi}.

\begin{lemma}\label{lem: PhiDecomposition}
     There exist $\phi_1,\phi_2,\beta \in \bbF_p[\lambda]$ with $\operatorname{gcd}(\phi_1,\phi_2) = 1$ such that   
    \begin{equation}
        \Phi_{\bX,\zeta}(\lambda,t) \equiv \beta(\lambda)\bigl(\phi_1(\lambda) + t \phi_2(\lambda) \bigr) \mod p.   \label{eq:QuickRug}
    \end{equation} 
\end{lemma}
\begin{proof}
    Expanding the determinant in the definition of $\Phi_{\bX,\zeta}$ yields that  
    \begin{equation}
        \Phi_{\bX,\zeta}(\lambda,t) = \sum_{i=0}^n \sum_{j=0}^n c_{i,j} \lambda^i t^j, \label{eq:OilyHome}
    \end{equation}
    for certain $c_{i,j} \in \bbZ$. 
    In fact, computing the determinant in an orthonormal basis containing $\zeta/\Vert \zeta \Vert$ by Laplace expansion shows that that $c_{i,j} = 0$ for $j\geq 2$.  
    (This also follows from \eqref{eq:TallHome}.)
    Now reduce \eqref{eq:OilyHome} modulo $p$ and factor out $\beta(\lambda)$ as the greatest common divisor of $\sum_{i=0}^n c_{i,0} \lambda^j$ and $\sum_{i=0}^n c_{i,1} \lambda^j$ to find \eqref{eq:TallHome}. 
\end{proof}
\begin{lemma}\label{lem: BetaSingleRoot}
    Suppose that $p\mid \ell$ and $\rank_p(\bW_{\bX,\zeta}) = n-1$. 
    Then there exists $\lambda_0 \in \bbF_p$ with $\beta(\lambda_0) \equiv 0\bmod p$ and there does not exist any $\lambda_1 \neq \lambda_0 $ with $\beta(\lambda_1) \equiv 0\bmod p$.  
\end{lemma}
\begin{proof}
    Recall from \eqref{eq: QWW} that $\bQ \bW_{\bX, \zeta} = \bW_{\bY, \eta}$. 
    Hence, using that $p\mid \ell$, it holds that 
    $
         \bW_{\bX,\zeta}^{\T}(\ell\bQ)^{\T}  \equiv 0 \bmod p.
    $    
    Therefore, $\rank_p(\ell\bQ^{\T})=1$ by the assumption that $\rank_p(\bW_{\bX,\zeta}) = n-1$ and the minimality of $\ell$.
    Thus, there exists $z\in \bbF_p^n \setminus \{0 \} $ with 
    \begin{equation}
        \bigl\{\lambda z: \lambda \in \bbF_p \bigr\} =\bigl\{(\ell \bQ^{\T})(v):v\in \bbF_p^n \bigr\} = \bigl\{q\in \bbF_p^n : \bW_{\bX,\zeta}^{\T}q =0 \bigr\}. \label{eq:GhostlyAnt}
    \end{equation}
    The vector $\bX z$ is again in the image of $\ell \bQ^{\T}$ since $\bX \bQ^{\T} = \bQ^{\T} \bY$.   
    Hence, $\bX z = \lambda_0 z $ for some $\lambda_0 \in \bbF_p$.
    Moreover, since $\zeta$ is a column of $\bW_{\bX,\zeta}$, it follows from \eqref{eq:GhostlyAnt} that $z^{\T} \zeta=0$ in $\bbF_p$ and hence $(\bX  + t \zeta \zeta^{\T} )z = \lambda_0 z $  for all $t$.     
    Thus, $\Phi_{\bX, \zeta}(\lambda_0, t) \equiv 0 \bmod p$ for every $t$. 
    This implies that $\beta(\lambda_0) = 0$ by using that $\operatorname{gcd}(\phi_1, \phi_2)=1$ in \eqref{eq:QuickRug}.     
     
    It remains to show that $\beta$ does not admit other roots in $\bbF_p$. 
    Suppose to the contrary that $\beta(\lambda_1) \equiv 0 \bmod p$ for $\lambda_1 \neq \lambda_0$.
    Then, $\Phi_{\bX,\zeta}(\lambda_1,t) \equiv 0 \bmod p$ for every $t$. 
    Hence, for every $t$ there exists $v_t\in \bbF_p^n \setminus \{0 \}$ such that  
    \begin{equation}
        (\bX + t\zeta \zeta^{\T}) v_t = \lambda_1 v_t.\label{eq:PricklyLion}
    \end{equation} 
    Using this for $t \in \{0,1 \}$ yields $\zeta \zeta^{\T} v_1 = (\lambda_1 \bI - \bX) v_1$ and $v_0^{\T} (\lambda_1 \bI -\bX) = 0$. 
    Hence, 
   \begin{equation}
        (v_0^{\T}\zeta )(\zeta^{T} v_1) = v_0^{\T}(\zeta \zeta^{T} v_1)  =  v_0^{\T} (\lambda_1 \bI - \bX) v_1 = 0. 
   \end{equation} 
   Thus, $\zeta^{\T} v_t = 0$ for some $t \in \{0,1 \}$.
   However, this yields $\bW_{\bX,\zeta}^{\T} v_t = 0$ by \eqref{eq:PricklyLion} and the definition \eqref{eq: Def_WXzeta} of $\bW_{\bX,\zeta}$. 
   Then $v_t$ is a multiple of $z$ by \eqref{eq:GhostlyAnt}, a contradiction since these are eigenvectors of $\bX + t\zeta \zeta^{\T}$ with different eigenvalues and hence linearly independent.
   Thus, $\lambda_1 = \lambda_0$, concluding the proof.   
\end{proof}
\begin{lemma}\label{lem: SingleRoot}
    Suppose that $p\mid \ell$ and $\rank_p(\bW_{\bX,\zeta}) = n-1$. 
    Then, there exists some $t_0 \in \bbF_p$ such that $\varphi_{\sX_{t_0}}(\lambda)$ has exactly one root in $\bbF_p$. 
\end{lemma}
\begin{proof}
    Let $\lambda_0$ be as in Lemma \ref{lem: BetaSingleRoot}. 
    Then, using that $\varphi_{\sX_{t_0}}(\lambda) = \Phi_{\bX,\zeta}(\lambda, t_0)$, it suffices to show that there exists some $t_0\in \bbF_p$ such that $\phi_1(\lambda) + t_0 \phi_2(\lambda) \neq 0$ for all $\lambda \neq \lambda_0$. 

    Suppose to the contrary that for every $t\in \bbF_p$ there exists $\Lambda_t\in \bbF_p\setminus \{\lambda_0 \}$ with $\phi_1(\Lambda_t)  + t \phi_2(\Lambda_t)=0$. 
    Then, by the pidgeonhole principle, there exist $t_1\neq t_2$ with $\Lambda_1 = \Lambda_2$. 
    However, this implies that $\phi_1(\Lambda_1) = 0$ and $\phi_2(\Lambda_1) =0$ contradicting the definition that $\operatorname{gcd}(\phi_1,\phi_2) = 1$ in Lemma \ref{lem: PhiDecomposition}. 
    This proves the desired result. 
\end{proof}

\subsection{Proof of the sufficient condition} \label{sec: Sufficient}
It remains to combine the preceding:
\begin{proposition}\label{prop: LevelPBound}
    Consider some prime $p$ such that $p^2 \nmid \det(\bW_{\bX,\zeta})$. 
    Then, $p\nmid \ell$.  
\end{proposition}
\begin{proof}
    Assume to the contrary that $p\mid \ell$ and $p^2 \nmid \det(\bW_{\bX,\zeta})$. 
    Due to Lemma \ref{lem: EllDividesDn} it is sufficient to consider the case where $p\mid \det(\bW_{\bX,\zeta})$. 
     
    Then, Lemma \ref{lem: PSquareRank} implies that $\bW_{\bX,\eta}$ and $\bW_{\bY,\eta}$ both have rank $n-1$ over $\bbF_p$.
    Lemma \ref{lem: SingleRoot} then yields some $t_0$ such that $\varphi_{\sX_{t_0}}(\lambda)$ has exactly one root in $\bbF_p$, ensuring that the assumption of Corollary \ref{cor: t0Roots} is satisfied so that $\bW_{\sX_{t_0}, \zeta}V \equiv 0 \bmod p$ and $\bW_{\sY_{t_0}, \eta}V \equiv 0 \bmod p$ have the same solutions $V$. 
    Lemma \ref{lem: StrongerLevelBound} now yields that $\ell \mid \det(\bW_{\sX_t, \zeta})/p$ since the orthogonal matrix $\bQ$ also satisfies $\bQ \sX_t \bQ^{\T} = \sY_t$ and $\bQ \zeta = \eta$. 
    We conclude that $p\nmid \ell$ since $\det(\bW_{\sX_t, \zeta}) = \det(\bW_{\bX,\zeta})$.   
\end{proof}
Note that the following result implies Theorem \ref{thm: WithLoops} as a special case: 
\begin{theorem}\label{thm: SqfreeIntegral}
    Suppose that $\det(\bW_{\bX,\zeta})$ is square free. 
    Then, $(\bX,\zeta)$ is characterized up to signed permutation by the $\Phi$-spectrum. 
\end{theorem}
\begin{proof}
    Note that a rational orthogonal matrix $\bQ$ has level one if and only if it is a signed permutation matrix, and note that $\ell = 1$ if and only if $p\nmid \ell$ for every prime $p$.
    The result is hence immediate from Proposition \ref{prop: LevelPBound}.  
\end{proof}
\begin{remark}\label{rem: WangYuDiscriminant}
    Wang and Yu established in \cite[Theorem 3.3]{wang2016square} that every prime factor $p$ of the level of a rational orthogonal matrix $\bQ$ with the property that $\bQ \bX \bQ^{T}\in \bbZ^{n\times n}$ has to satisfy $p \mid \Delta_{\bX}$ where $\Delta_{\bX}\in \bbZ$ is the discriminant of the characteristic polynomial $\varphi_{\bX}(\lambda) = \det(\lambda \bI - \bX)$.
    Moreover, by \cite[Lemma 4.5]{wang2016square} it holds that $p^2 \mid \Delta_{\bX}$ if $p$ is odd. 
    One can combine this with Proposition \ref{prop: LevelPBound} for a stronger sufficient condition: $(\bX,\zeta)$ is characterized by $\Phi$-spectrum whenever $\gcd\{\Delta_{\bX}, \det(\bW_{\bX,\zeta}) \}$ is square free and additionally $2^2 \nmid \det(\bW_{\bX,\zeta})$ or $2\nmid \Delta_{\bX}.$
\end{remark}

\section{Conjecture for the satisfaction frequency}\label{conj: Satisfaction}
As mentioned in the introduction, \Cref{thm: SqfreeIntegral} is convenient for probabilistic investigations, which was our main motivation for the present paper. 
In particular, analytically tractable random matrix models are analyzed in our forthcoming work \cite{vanwerde2025}. 
We believe that those models are in the same universality class as graphs with loops, leading to the following Conjecture \ref{conj: WalkMatrixSquareFree}:

\begin{conjecture}[{Lvov and Van Werde \cite{vanwerde2025}}]\label{conj: WalkMatrixSquareFree}
    Suppose that $\bX$ is chosen uniformly at random from all symmetric $n\times n$ matrices with $\{0,1 \}$-valued entries and consider an independent random vector $\zeta$ that is uniform on $\{0,1 \}^n$.   
    Then,
    \begin{align}
        \lim_{n\to \infty}\bbP\bigl(\det(\bW_{\bX,\zeta})\textnormal{ is square-free}\bigr) &= \prod_{p\textnormal{ prime}} \Bigl(1 - \frac{1}{p^2} - \frac{1}{p^3} + \frac{1}{p^4}\Bigr) \prod_{k=1}^\infty \Bigl(1 - \frac{1}{p^{2k}}\Bigr) \nonumber \\ 
        &= 0.2943\ldots \label{eq: DetWalkSquarefree}
    \end{align}
\end{conjecture}

Similar conjectures are expected to be more difficult for graphs without loops, or if $\zeta$ is taken to be the all-ones vector. 
In particular, additional complications are necessary in that setting related to special behavior for the prime $p=2$.

\subsection*{Acknowledgement}
I thank Nikita Lvov for discussions related to the topic of this paper. 
Funded by the Deutsche Forschungsgemeinschaft (DFG, German Research Foundation) under Germany's Excellence Strategy EXC 2044 –390685587, Mathematics M\"unster: Dynamics--Geometry--Structure.

{
\bibliographystyle{abbrv}

\begin{thebibliography}{10}

\bibitem{abiad2024switching}
A.~Abiad, N.~van~de Berg, and R.~Simoens.
\newblock Switching methods of level 2 for the construction of cospectral graphs.
\newblock {\em arXiv preprint arXiv:2410.07948}, 2024.
\newblock \doi{10.48550/arXiv.2410.07948}.

\bibitem{coutinho2021graph}
G.~Coutinho and C.~Godsil.
\newblock Graph spectra and continuous quantum walks, 2021.
\newblock \href{https://www.math.uwaterloo.ca/~cgodsil/quagmire/QuantumWalks/pdfs/GrfSpc3.pdf}{\nolinkurl{https://www.math.uwaterloo.ca/~cgodsil/quagmire/QuantumWalks/pdfs/GrfSpc3.pdf}}.

\bibitem{doob2002complement}
M.~Doob and W.~Haemers.
\newblock The complement of the path is determined by its spectrum.
\newblock {\em Linear algebra and its applications}, 2002.
\newblock \doi{10.1016/S0024-3795(02)00323-3}.

\bibitem{godsil1982constructing}
C.~Godsil and B.~McKay.
\newblock Constructing cospectral graphs.
\newblock {\em {A}equationes {M}athematicae}, 1982.
\newblock \doi{10.1007/BF02189621}.

\bibitem{gunthard1956zusammenhang}
H.~G{\"u}nthard and H.~Primas.
\newblock {Z}usammenhang von graphentheorie und {M}{O}-theorie von molekeln mit systemen konjugierter bindungen.
\newblock {\em Helvetica Chimica Acta}, 1956.
\newblock \doi{10.1002/hlca.19560390623}.

\bibitem{haemers2016almost}
W.~H. Haemers.
\newblock Are almost all graphs determined by their spectrum.
\newblock {\em Notices of the South African Mathematical Society}, 2016.

\bibitem{haemers2008spectral}
W.~H. Haemers, X.~Liu, and Y.~Zhang.
\newblock Spectral characterizations of lollipop graphs.
\newblock {\em Linear Algebra and its Applications}, 2008.
\newblock \doi{10.1016/j.laa.2007.10.018}.

\bibitem{johnson1980note}
C.~Johnson and M.~Newman.
\newblock A note on cospectral graphs.
\newblock {\em Journal of Combinatorial Theory, Series B}, 1980.
\newblock \doi{10.1016/0095-8956(80)90058-1}.

\bibitem{kac1966can}
M.~Kac.
\newblock Can one hear the shape of a drum?
\newblock {\em The {A}merican mathematical monthly}, 1966.
\newblock \doi{10.1080/00029890.1966.11970915}.

\bibitem{vanwerde2025}
N.~Lvov and A.~Van~Werde.
\newblock In preparation.

\bibitem{miao2026smith}
S.~Miao, S.~Li, and L.~Zhang.
\newblock Smith normal form and the generalized spectral characterization of a type of symmetric matrices.
\newblock {\em Applied Mathematics and Computation}, 2026.
\newblock \doi{10.1016/j.amc.2025.129764}.

\bibitem{qiu2023generalized}
L.~Qiu, Y.~Ji, L.~Mao, and W.~Wang.
\newblock Generalized spectral characterizations of regular graphs based on graph-vectors.
\newblock {\em Linear Algebra and its Applications}, 2023.
\newblock \doi{10.1016/j.laa.2023.01.006}.

\bibitem{qiu2023smith}
L.~Qiu, W.~Wang, and H.~Zhang.
\newblock Smith normal form and the generalized spectral characterization of graphs.
\newblock {\em {D}iscrete {M}athematics}, 2023.
\newblock \doi{10.1016/j.disc.2022.113177}.

\bibitem{schwenk1973almost}
A.~Schwenk.
\newblock Almost all trees are cospectral.
\newblock {\em New directions in the theory of graphs}, 1973.
\newblock \href{https://www.researchgate.net/publication/245264768_Almost_all_trees_are_cospectral}{\nolinkurl{https://www.researchgate.net/publication/245264768_Almost_all_trees_are_cospectral}}.

\bibitem{shen2005graph}
X.~Shen, Y.~Hou, and Y.~Zhang.
\newblock Graph {$Z_n$} and some graphs related to {$Z_n$} are determined by their spectrum.
\newblock {\em Linear Algebra and its Applications}, 2005.
\newblock \doi{10.1016/j.laa.2005.01.036}.

\bibitem{van2003graphs}
E.~van {D}am and W.~{H}aemers.
\newblock {W}hich graphs are determined by their spectrum?
\newblock {\em {L}inear {A}lgebra and its {A}pplications}, 2003.
\newblock \doi{10.1016/S0024-3795(03)00483-X}.

\bibitem{wang2013generalized}
W.~Wang.
\newblock {G}eneralized spectral characterization of graphs revisited.
\newblock {\em {T}he {E}lectronic {J}ournal of {C}ombinatorics}, 2013.
\newblock \doi{10.37236/3748}.

\bibitem{wang2017simple}
W.~Wang.
\newblock A simple arithmetic criterion for graphs being determined by their generalized spectra.
\newblock {\em {J}ournal of {C}ombinatorial {T}heory, {S}eries {B}}, 2017.
\newblock \doi{10.1016/j.jctb.2016.07.004}.

\bibitem{wang2019cospectral}
W.~Wang, L.~Qiu, and Y.~Hu.
\newblock Cospectral graphs, {G}{M}-switching and regular rational orthogonal matrices of level $p$.
\newblock {\em Linear Algebra and its Applications}, 2019.
\newblock \doi{10.1016/j.laa.2018.10.027}.

\bibitem{wang2006sufficient}
W.~Wang and C.-X. Xu.
\newblock A sufficient condition for a family of graphs being determined by their generalized spectra.
\newblock {\em European Journal of Combinatorics}, 2006.
\newblock \doi{10.1016/j.ejc.2005.05.004}.

\bibitem{wang2016square}
W.~Wang and T.~Yu.
\newblock Square-free discriminants of matrices and the generalized spectral characterizations of graphs.
\newblock {\em arXiv preprint arXiv:1608.01144}, 2016.
\newblock \doi{10.48550/arXiv.1608.01144}.

\bibitem{wang2024graph}
W.~Wang and D.~Zhao.
\newblock Graph isomorphism and multivariate graph spectrum.
\newblock {\em Advances in Applied Mathematics}, 2026.
\newblock \doi{10.1016/j.aam.2025.102994}.

\end{thebibliography}

}

\end{document}